\documentclass[a4paper,11pt]{amsart}

\usepackage{epsfig}
\usepackage{amsthm,amsfonts}
\usepackage{amssymb,graphicx,color}
\usepackage{float}
\graphicspath{ {Images/} }
\usepackage[labelfont={bf},textfont={it}]{caption}
\usepackage[labelfont={bf, it}]{subcaption}
\usepackage[all]{xy}
\usepackage{verbatim}
\usepackage{hyperref}
\usepackage{enumitem}
\usepackage{tikz-cd}

\newtheorem{theorem}{Theorem}[section]
\newtheorem*{theorem*}{Theorem}

\newtheorem{corollary}[theorem]{Corollary}
\newtheorem{proposition}[theorem]{Proposition}
\newtheorem{definition}[theorem]{Definition}

\newtheorem{problem}{Problem}
\newtheorem{claim}{Claim}[theorem]
\newtheorem{example}[theorem]{Example}
\newtheorem{remark}[theorem]{Remark}

\newtheorem{innercustomthm}{Theorem}
\newenvironment{customthm}[1]
  {\renewcommand\theinnercustomthm{#1}\innercustomthm}
  {\endinnercustomthm}

\newcommand{\R}{\mathbb{R}}
\newcommand{\C}{\mathbb{C}}

\begin{document}

\title[Log-Lipschitz and H\"older regularity imply smoothness]
{Log-Lipschitz and H\"older regularity imply smoothness for complex analytic sets}

\author[J. E. Sampaio]{Jos\'e Edson Sampaio}

\address{Jos\'e Edson Sampaio: Departamento de Matem\'atica, Universidade Federal do Cear\'a,
	      Rua Campus do Pici, s/n, Bloco 914, Pici, 60440-900, 
	      Fortaleza-CE, Brazil. E-mail: {\tt edsonsampaio@mat.ufc.br} 	   
}

\thanks{The author was partially supported by CNPq-Brazil grant 310438/2021-7. This work was supported by the Serrapilheira Institute (grant number Serra -- R-2110-39576).
}


\begin{abstract}
In this paper, we prove metric analogues, in any dimension and in any co-dimension, of the famous Theorem of Mumford on smoothness of normal surfaces and the beautiful Theorem of Ramanujam that gives a topological characterization of $\mathbb{C}^2$ as an algebraic surface. For instance, we prove that a complex analytic set that is log-Lipschitz regular at 0 (i.e., a complex analytic set that has a neighbourhood of the origin which bi-log-Lipschitz homeomorphic to an Euclidean ball) must be smooth at 0. We prove even more, we prove that if a complex analytic set $X$ such that, for each $0<\alpha<1$, $(X,0)$ and $(\mathbb{R}^k,0)$ are bi-$\alpha$-H\"older homeomorphic, then $X$ must be smooth at 0. These results generalize the Lipschitz Regularity Theorem, which says that a Lipschitz regular complex analytic set must be smooth. Global versions of these results are also presented here and, in particular, we obtain a characterization of an affine linear subspace as a pure-dimensional entire complex analytic set.
\end{abstract}

\maketitle
\tableofcontents

\section{Introduction}

In 1961, D. Mumford in his famous article \cite{Mumford:1961} proved that {\it a complex algebraic surface, with normal singularity at $x_0$, is smooth at $x_0$ if and only if its link at $x_0$ is simply connected}. Since in the case of a complex normal surface $X$, its link at $x_0\in X$ is simply connected if and only if $X$ is topologically regular at $x_0$ (i.e., there is an open neighbourhood $U$ of $x_0$ such that $X\cap U$ is homeomorphic to some open euclidean ball), we can state the Theorem of Mumford as the following:
\begin{theorem}[Theorem of Mumford]
Let $X\subset \C^n$ be a complex algebraic set with a normal singularity at $x_0$. Then $X$ is smooth at $x_0$ if and only if $X$ is topologically regular at $x_0$. 
\end{theorem}
Unfortunately, this result does not hold true in dimensions greater than 2, as it was proved by F. Brieskorn in \cite{Brieskorn:1966} (see also \cite{Brieskorn:1966b}). For instance, he proved that the hypersurface $X=\{(x,y,z,w)\in \C^4;x^2+y^2+z^2=w^3\}$, which has a normal singularity at $0$, is topologically regular at $0$, but is not smooth at $0$. Thus, in order to obtain a result like the Theorem of Mumford in higher dimension, we have to impose more conditions. As another reference of works in the same line, we also have D. Prill \cite{Prill:1967} where it was proved that {\it any d-dimensional complex algebraic cone with trivial i-th homotopy group of the link, for all $0\leq i \leq 2d-2$, must be a linear affine subspace}. One more result of the type of  Theorem of Mumford was obtained separately by N. A'Campo in \cite{Acampo:1973} and L\^e D. T. in \cite{Le:1973}, which can be stated as follows:
{\it
Let $(X,x_0)$ be a codimension one germ of complex analytic subset of $\mathbb{C}^n$. If there is a homeomorphism $\varphi\colon(\mathbb{C}^n,X,x_0)\to (\mathbb{C}^n,\mathbb{C}^{n-1}\times\{0\},x_0)$, then $X$ is smooth at $x_0$.
}

However, the Theorem of A'campo-L\^e does not hold, in general, for complex analytic sets that are topological submanifolds in codimension greater than 1. For instance, the cusp $C=\{(x,y,z)\in \C^3;x^3=y^2$ and $z=0\}$ is a topological submanifold of $\C^3$, but is not smooth at $0$.

Recently, it was presented two different approaches for a type of Theorem of Mumford in higher dimension: one from the contact geometry point of view and the other one from the Lipschitz geometry point of view. 

From the contact geometry point of view, M. McLean in \cite{McLean:2016} showed that: {\it if a complex analytic set $A$ has a normal isolated singularity at $0$ and its link at $0$ (with its canonical contact structure) is contactomorphic to the link of $\mathbb{C}^3$ (the standard contact sphere), then $A$ is smooth at $0$}.  

From the Lipschitz geometry point of view,  the author in his thesis \cite{Sampaio:2015} (see also \cite{Sampaio:2016} and \cite{BirbrairFLS:2016}) proved the following:
\begin{theorem}[Lipschitz Regularity Theorem]
A complex analytic set $X$ is smooth at $x_0$ if and only if $X$ is, around $x_0$, bi-Lipschitz homeomorphic to an open Euclidean ball.
\end{theorem}

Around a decade ago the author was presented, in a particular conversation, to the following problem:

\begin{problem}\label{conjecture_one}
Let $X\subset\mathbb{C}^n$ be a complex analytic set of dimension $d$. If $X$ is log-Lipschitz regular at $0$, then is $X$ smooth at $0$?
\end{problem}

In order to know, $h$ is {\bf log-Lipschitz} if there exists a positive constant $C$ such that $\|h(x)-h(y)\|\leq C\|x-y\|\big|\log \|x-y\|\big|$ for all $ x, y\in X$ such that $0<\|x-y\|<1/2$. Thus, we say that $X$ is log-Lipschitz regular at $0$ if there are an open neighbourhood $U$ of $0$ and a homeomorphism $h\colon X\cap U\to B_r^k(0)$ such that $h(0)=0$ and $h$ and $h^{-1}$ are log-Lipschitz. Similarly, we define the notions of $\alpha$-H\"older and Lipschitz regularities at $0$ (see the definitions of the different regularities in Subsection \ref{subsec:reg_def}).

In this article, we answer affirmatively Problem \ref{conjecture_one}. Indeed, we obtain the following result of smoothness of analytic sets, which gives even more than a positive answer to Problem \ref{conjecture_one}.
\begin{customthm}{\ref*{main_result}}
Let $X\subset\mathbb{C}^n$ be a complex analytic set. Then the following statements are equivalent:
\begin{itemize}
 \item [(1)] $X$ is $\alpha$-H\"older regular at $0$ for all $\alpha\in (0,1)$;
 \item [(2)] $X$ is strongly H\"older regular at $0$;
 \item [(3)] $X$ is log-Lipschitz regular at $0$;
 \item [(4)] $X$ is Lipschitz regular at $0$;
 \item [(5)] $X$ is smooth at $0$.
\end{itemize}
\end{customthm}

Note that in the very particular case when $\dim X=1$, Theorem \ref{main_result} is a consequence of the main result in \cite{FernandesSS:2018} (see some related studies in \cite{FernandesFSS:2024} and \cite{FernandesSS:2024}).

Theorem \ref{main_result} is sharp in the following sense: for any $\alpha\in (0,1)$ and any positive integer $d$, there is a complex algebraic hypersurface $X\subset \C^{d+1}$ that is $\alpha$-H\"older regular at $0$, but is not smooth at $0$ (see Proposition \ref{prop:example_sharp}).

The following example shows that one cannot obtain a result like Theorem \ref{main_result} for sets in general.
\begin{example}\label{example:log-regular}
Let $f\colon \R\to \R$ be the function given by
$$
f(x)=\left\{\begin{array}{ll}
            |x||\log |x||,& \mbox{ if }x\not=0,\\
            0,& \mbox{ if }x=0.
            \end{array}\right.
$$
Then $X=Graph(f)$ is log-Lipschitz regular at $0$, but is not Lipschitz regular at $0$. Indeed, the restriction of the projection on the first coordinate $\pi\colon X\to \R$ is a bi-log-Lipschitz homeomorphism.
\end{example}

It is clear that Theorem \ref{main_result} is a generalization of the Lipschitz Regularity Theorem. However, Example \ref{example:log-regular} also shows that the proof of the Lipschitz Regularity Theorem cannot be used to prove Theorem \ref{main_result}.
Indeed, the main ingredients in the proof of the Lipschitz Regularity Theorem were the following: 
\begin{itemize}
 \item [(1)] The bi-Lipschitz invariance of the tangent cones in \cite[Teorema 4.1.1]{Sampaio:2015} (see also  \cite[Theorem 3.2]{Sampaio:2016}, \cite[Theorem 3.6]{KoikeP:2019} and \cite[Theorem 3.1]{SampaioS:2022}), which says that if two germs of definable sets in an o-minimal structure on $\R$ are bi-Lipschitz homeomorphic, then their tangent cones are bi-Lipschitz homeomorphic too (see the definition of tangent cone in Definition \ref{def:tg_cone}); 
 \item [(2)] Lipschitz regularity implies Lipschitz normally embeddedness, i.e., a set that is Lipschitz regular at $0$ must be Lipschitz normally embedded (or quasi-convex) at $0$ (see Definition \ref{def:lne}).
\end{itemize}
However, there are germs of sets $(A,a)$ and $(B,b)$ that are bi-log-Lipschitz homeomorphic, but their tangent cones are not homeomorphic. For instance, in Example \ref{example:log-regular}, $C(X,0)=\{(0,y)\in \R^2; y\geq 0\}$ and $C(\R,p)=\R$, for any $p\in\R$. Moreover, there are sets that are log-Lipschitz regular at $0$, but are not Lipschitz normally embedded at $0$. For instance, the set $X$ in Example \ref{example:log-regular} is log-Lipschitz regular at $0$, but is not Lipschitz normally embedded at $0$.

In order to know more about the o-minimal geometry, see, for instance, \cite{Coste:1999} and \cite{Dries:1998}. 

Still in this article, we also present a global version of Theorem \ref{main_result}, which can be seen as a Lipschitz analogous of the following beautiful topological characterization of $\C^2$ as an algebraic surface, presented by Ramanujam in \cite[Theorem]{Ramanujam:1971}.
\begin{theorem}[Theorem of Ramanujam]
Let $X$ be an affine smooth complex algebraic surface that is contractible and simply connected at infinity. Then $X$ is isomorphic to $\C^2$ as an algebraic variety. 
\end{theorem}

Nowadays, the Theorem of Ramanujam is equivalent to say that {\it any affine smooth complex algebraic surface $X\subset \C^n$, which is contractible and $X\setminus B_r^{2n}(0)$ is homeomorphic to $\C^2\setminus B_r^4(0)$ for all big enough $r>0$, must be isomorphic to $\C^2$ as an algebraic variety}.

Another result in this line was presented by Ahern and Rudin in \cite{AhernR:1993}.
To be more precise, Ahern and Rudin in \cite{AhernR:1993} defined the notion of a set to be $C^1$-smooth at infinity.

\begin{definition}
A set $V\subset \R^n$ is {\bf $C^1$-smooth at infinity} if $\iota (V\setminus\{0\})\cup \{0\}$ is a $C^1$ submanifold around $0$, where $\iota\colon \R^n\setminus \{0\}\to \R^n\setminus \{0\}$ is the mapping given by $\iota(x)=\frac{x}{\|x\|^2}$.
\end{definition}
With this notion, Ahern and Rudin in \cite{AhernR:1993} proved the following very interesting result:
\begin{theorem}[Thereom II in \cite{AhernR:1993}]\label{thm:Ahern-Rudin}
Let $V\subset \mathbb{C}^n$ be a complex analytic set. Then $V$ is $C^1$ smooth at infinity if and only if $V$ is the union of an affine linear subspace of $\mathbb{C}^n$ and a (possibly empty) finite set.
\end{theorem}

Similarly, we define also the following:
\begin{definition}
A set $V\subset \R^n$ is:
\begin{itemize}
 \item {\bf $\alpha$-H\"older smooth at infinity} if $\iota (V\setminus\{0\})\cup \{0\}$ is $\alpha$-H\"older regular at $0$;
 \item {\bf strongly H\"older smooth at infinity} if $\iota (V\setminus\{0\})\cup \{0\}$ is strongly H\"older regular at $0$;
 \item {\bf log-Lipschitz smooth at infinity} if $\iota (V\setminus\{0\})\cup \{0\}$ is log-Lipschitz regular at $0$;
 \item {\bf Lipschitz smooth at infinity} if $\iota (V\setminus\{0\})\cup \{0\}$ is Lipschitz regular at $0$.
\end{itemize}
\end{definition}

Thus, for algebraic sets, the next result generalizes \cite[Thereom II]{AhernR:1993} and \cite[Theorem 3.8]{FernandesS:2020} (see also \cite[Theorem 1.6]{Sampaio:2023}, \cite[Corollary 4.3]{FernandesS:2023}, \cite[Theorem 3.1]{Sampaio:2019} and \cite[Corollary 2.13]{Sampaio:2020}):
\begin{customthm}{\ref*{thm:gen_Ahern-Rudin}}
Let $V\subset \mathbb{C}^n$ be a complex algebraic set. Then the following statements are equivalent:
\begin{itemize}
 \item [(1)] $V$ is $\alpha$-H\"older smooth at infinity for all $\alpha\in (0,1)$;
 \item [(2)] $V$ is strongly H\"older smooth at infinity;
 \item [(3)] $V$ is log-Lipschitz smooth at infinity;
 \item [(4)] $V$ is Lipschitz smooth at infinity;
 \item [(5)] $V$ is the union of an affine linear subspace of $\mathbb{C}^n$ and a (possibly empty) finite set.
\end{itemize}
\end{customthm}

We also present a generalization of  \cite[Thereom II]{AhernR:1993} for complex analytic sets in general. Indeed, we prove that {\it a complex analytic set $V\subset \mathbb{C}^n$ is Lipschitz smooth at infinity if and only if $V$ is the union of an affine linear subspace of $\mathbb{C}^n$ and a (possibly empty) finite set} (see Corollary \ref{cor:AhernR_Lip}). This gives a Lipschitz characterization of $\C^d$ as an affine linear subspace among all the pure-dimensional entire complex analytic sets.

The main ingredients in our proofs are the use of the new metric algebraic topology theory called {\bf moderately discontinuous homology} introduced in \cite{FHPS} (see Subsection \ref{subsec:mdh}) and
the use of a distance called here {\bf diameter distance}.
The diameter distance on a path-connected set $X$, which is defined as follows: given two points $x_1,x_2\in X$, $d_{X,diam}(x_1,x_2)$  is the infimum of the diameters of the image of paths on $X$ connecting $x_1$ to $x_2$.

In general, one has that $|x_1-x_2|\leq d_{X,diam}(x_1,x_2)\leq d_{X,inn}(x_1,x_2)$, for all pair of points $x_1,x_2\in X$, where $d_{X,inn}$ is the
{\bf inner distance} on $X$, which is defined as follows: given two points $x_1,x_2\in X$, $d_{X,inn}(x_1,x_2)$  is the infimum of the lengths of paths on $X$ connecting $x_1$ to $x_2$.

An important class of sets where these distances are equivalent are the LNE sets. In order to know, we say that a path connected set $X\subset\R^N$ is {\bf Lipschitz normally embedded (LNE)} if there exists a constant $C\geq 1$ such that $d_{X, inn}(x_1,x_2)\leq C\|x_1-x_2 \|$, for all pair of points $x_1,x_2\in X$.

However, these three distances are not equivalent in general, as it is shown in the next example:

\begin{example}
Let $G$ be the graph of the function $f\colon \R\to \R$ given by
$$
f(x)=\left\{\begin{array}{ll}
            x\sin (\frac{1}{x}),& \mbox{ if }x\not=0,\\
            0,& \mbox{ if }x=0.
            \end{array}\right.
$$
Let $X=\{(tx,ty,t)\in \R^3;t\geq 0$ and $(x,y)\in G\}$. For any point $(x,y)\in G\setminus\{(0,0)\}$, we have $d_{X,inn}((x,y,1),(0,0,1))=1+|(x,y)|\geq 2$. However, $d_{X,diam}((x,y,1),(0,0,1))\to 0$ as $(x,y)\to (0,0)$. Thus, $d_{X,diam}$ is not equivalent to $d_{X,inn}$.
\end{example}

In the case of the (real) cusp $C=\{(x,y)\in\R^2;y^3=x^2\}$, which is a classical example of a set that is not LNE, we have that $d_{C,inn}$ and $d_{C,diam}$ are equivalent distances.
Thus, we have the following natural problem:

\begin{problem}\label{conjecture:diameter_inner_equivalence}
Let $X\subset \R^n$ be a path-connected set. Assume that $X$ is a definable set in an o-minimal structure on $\R$. Is $d_{X,diam}$ equivalent to $d_{X,inn}$?
\end{problem}

In this article, we give a positive answer to Problem \ref{conjecture:diameter_inner_equivalence} (see Theorem \ref{thm:diameter_inner_equivalence}).

\bigskip

\noindent{\bf Acknowledgements}. The author would like to thank Ederson Braga for presenting the problem on log-Lipschitz regularity of complex analytic sets.

\section{Preliminaries}\label{sec:preliminaries}

All the subsets of $\R^n$ or $\mathbb{C}^n$ considered in the paper are supposed to be equipped with the Euclidean distance. When we consider other distance, it is clearly emphasized.

\noindent {\bf Notation:} 
\begin{itemize}
 \item $\|(x_1,...,x_n)\|=(x_1^2+...+x_n^2)^{\frac{1}{2}}$;
 \item $\mathbb{S}_r^{n-1}(p)=\{x\in \R^n; \|x-p\|=r\}$, $ \mathbb{S}_r^{n-1}=\mathbb{S}_r^{n-1}(0)$ and $ \mathbb{S}^{n-1}=\mathbb{S}_1^{n-1}(0)$;
 \item $B_r^{n}(p)=\{x\in \R^n; \|x-p\|<r\}$;
 \item For $X, Y\in \R^n$, $dist(X,Y):=\inf\{\|x-y\|;x\in X$ and $y\in Y\}$;
 \item Let $f,g\colon (0,\varepsilon)\to [0,+\infty)$ be functions. We write $f\lesssim g$ if there is a constant $f(t)\leq C g(t)$ for all $t\in (0,\varepsilon)$. We write $f\approx g$ if $f\lesssim g$ and $g\lesssim f$. We write $f\ll g$ if $\lim\limits_{t\to 0^+} \frac{f(t)}{g(t)}=0$.  
\end{itemize}

\subsection{Definitions of regularity}\label{subsec:reg_def}

\begin{definition}
Let $X\subset \R^n$ and $Y\subset \R^m$ be two sets and let $h\colon X\to Y$.
\begin{itemize}
 \item We say that $h$ is {\bf Lipschitz} (resp. {\bf $\alpha$-H\"older}) if there exists a positive constant $C$ such that $\|h(x)-h(y)\|\leq C\|x-y\|$ (resp. $\|h(x)-h(y)\|\leq C\|x-y\|^{\alpha}$) for all $ x, y\in X$. In this case, we say that $C$ is a {\bf Lipschitz} (resp. {\bf H\"older}) {\bf constant of $h$}.
 
\item We say that $h$ is {\bf bi-Lipschitz} (resp. {\bf bi-$\alpha$-H\"older}) if $h$ is a homeomorphism, it is Lipschitz (resp. $\alpha$-H\"older) and its inverse is also Lipschitz (resp. $\alpha$-H\"older). In this case, we also say that $X$ is bi-Lipschitz (resp. bi-$\alpha$-H\"older) homeomorphic to $Y$.

 \item We say that $h$ is {\bf log-Lipschitz} if there exists a positive constant $C$ such that $\|h(x)-h(y)\|\leq C\|x-y\|\big|\log \|x-y\|\big|$ for all $ x, y\in X$ such that $0<\|x-y\|<1/2$.
 
\item We say that $h$ is {\bf bi-log-Lipschitz} if $h$ is a homeomorphism, it is log-Lipschitz and its inverse is also log-Lipschitz. In this case, we also say that $X$ is bi-log-Lipschitz homeomorphic to $Y$.
\end{itemize}
\end{definition}

\begin{definition}
Let $X\subset \R^n$ be a set and $p\in X$.
\begin{itemize}
 \item We say that $X$ is {\bf $\alpha$-H\"older regular at $p$} if there exist $k\in \mathbb{N}$, $r>0$, an open neighbourhood $U\subset \R^n$ of $p$ and a bi-$\alpha$-H\"older homeomorphism $h\colon X\cap U\to B_r^k(0)$ such that $h(p)=0$.

 \item We say that $X$ is {\bf strongly H\"older regular at $p$} if there exist $k\in \mathbb{N}$, $r>0$, an open neighbourhood $U\subset \R^n$ of $p$ and a  homeomorphism $h\colon X\cap U\to B_r^k(0)$ that is a bi-$\alpha$-H\"older homeomorphism for all $\alpha\in (0,1)$ and such that $h(p)=0$.
 
\item We say that $X$ is {\bf Lipschitz} (resp. {\bf log-Lipschitz}) {\bf regular at $p$} if there exist $k\in \mathbb{N}$, $r>0$, an open neighbourhood $U\subset \R^n$ of $p$ and a bi-Lipschitz (resp. bi-log-Lipschitz) homeomorphism $h\colon X\cap U\to B_r^k(0)$ such that $h(p)=0$.
\end{itemize}
\end{definition}

\begin{definition}\label{def:lne}
Let $X\subset\R^N$ be a subset. We say that $X$ is {\bf Lipschitz normally embedded (LNE)} if there exists a constant $C\geq 1$ such that $d_{X, inn}(x_1,x_2)\leq C\|x_1-x_2 \|$, for all pair of points $x_1,x_2\in X$. In this case, we also say that $X$ is $C$-LNE. We say that $X$ is {\bf LNE at $p\in \R^N$}  if there exists an open neighborhood $U$ of $p$ such that $X\cap U$ is LNE. We say that $X$ is {\bf LNE at infinity} if there exists a compact subset $K\subset\mathbb{R}^N$ such that $X\setminus K$ is LNE.
\end{definition}

\subsection{Tangent cones}
\begin{definition}\label{def:tg_cone}
Let $X\subset \R^{n}$ be a subset and let $p\in \overline{X}$ (resp. $p=\infty$).
We say that $v\in \R^{n}$ is {\bf a tangent vector of $X$ at $p$} ( resp. {\bf infinity}) if there are a sequence $\{x_j\}_{j\in \mathbb{N}}\subset X$ and a sequence of positive real numbers $\{t_j\}_{j\in \mathbb{N}}$ such that $\lim\limits_{j\to \infty} t_j=0$ (resp. $\lim\limits_{j\to \infty} t_j=+\infty$) and $\lim\limits_{j\to \infty} \frac{1}{t_j}(x_j-p)=v$ (resp. $\lim\limits_{j\to \infty} \frac{x_j}{t_j}=v$). The set of all tangent vectors of $X$ at $p$ (resp. infinity) is denoted by $C(X,p)$ (resp. $C(X, \infty)$) and is called the {\bf tangent cone of $X$ at $p$} (resp. the {\bf tangent cone of $X$ at infinity}).
\end{definition}
\begin{remark}
Let $X\subset \mathbb{C}^n$ be a complex analytic (resp. algebraic) set and $x_0\in X$. In this case, $C(X,x_0)$ (resp. $C(X, \infty)$) is the zero set of a set of complex homogeneous polynomials, see this in \cite[Theorem 4D]{Whitney:1972} (resp. \cite[Theorem 1.1]{LeP:2016} or \cite[Theorem 3.1]{Sampaio:2023}). In particular, $C(X,x_0)$ (resp. $C(X, \infty)$) is the union of complex lines passing through the origin $0\in\mathbb{C}^n$.
\end{remark}

When $X\subset \R^N$ is a subanalytic (resp. semialgebraic) subset, we denote by $Link_p(X)$ (resp. $Link_{\infty}(X)$) to be the link of $X$ at $p$ (resp. infinity).

\subsection{Pancake decomposition}\label{subsec:pancake_dec}

Let $A\subset \mathbb{R}^n$ be a connected compact subanalytic set. By \cite[Proposition 3]{KurdykaO:1997}, given $\epsilon>0$, there is a partition into a subanalytic and finite union $A=\bigcup_{i\in I}{B_i}$ such that every subset $B_{i}$ is $(1+\epsilon)$-LNE. For each $i$, let $X_i$ be the closure of $B_i$ in $A$. We consider $x,y\in A$ and for each $r$ and $k$ we define  
$$\widetilde{\Delta}_r(x,y):=\inf\left\{\sum_{i=0}^{r-1}{\|x_{i+1}-x_i\|};x_0=x, x_r=y, x_i,x_{i+1}\in X_{\nu_i}, 0\leq i\leq r-1\right\}.$$

$$\Delta_k(x,y):=\inf\left\{\widetilde{\Delta}_r(x,y);r=1,\cdots,k\right\}.$$
We also define $\inf \emptyset=+\infty$. It is clear that every $\Delta_k$ is subanalytic. Finally, we define 
$$d_{A,P}(x,y):=\inf\left\{\Delta_k(x,y);k\in \mathbb{N} \right\}.$$ 

\begin{proposition}[Lemma 4 and Theorem 1 in \cite{KurdykaO:1997}]\label{prop:pancake_distance}
	The function $d_{A,P}\colon A\times A\rightarrow \mathbb{R}$ is subanalytic, defines a distance in $A$ and 
	$$d_{A,P}(x,y)\leq d_{A,inn}(x,y)\leq  (1+\epsilon)d_{A,P}(x,y),$$
for all $x,y\in A$.
\end{proposition}
The distance $d_{A,P}$ is called a {\bf pancake distance on $A$}.

\subsection{Moderately discontinuous homology}\label{subsec:mdh}
Since Moderately discontinuous homology is a very recent theory, we briefly recall the main definitions of it. For more details about this theory, we refer the reader to \cite{FHPS}. The theory works for a large class of distances, but here we only consider the outer distance $d_{out}$, that one induced by the Euclidean space.

For $n \in \mathbb N$, let $\Delta_n \subset \mathbb R^{n+1}$ denote the standard $n$-simplex (i.e, $\Delta_n :=\{(p_0, \ldots, p_n) \in \mathbb R_{\geq 0}^{n+1}: \sum_{i=0}^n p_i = 1\}$)
with the orientation induced by the standard orientation in $\mathbb R^{n+1}$ of the convex hull of $\Delta_n \cup 0$. For $0\leq k \leq n$, denote by $i_n^k: \Delta_{n-1} \mapsto \Delta_n, (p_0, \ldots, p_{n-1}) \mapsto (p_0, \ldots, p_{k-1}, 0, p_k, \ldots, p_{n-1})$. Set $$\hat{\Delta}_n:=\{(tx, t) \in \mathbb R^{n+1}\times \mathbb R: x \in \Delta_n, t \in [0,1)\}$$ and let $\hat{j}_n^k: \hat{\Delta}_{n-1} \to \hat{\Delta}_n, (tx,t) \mapsto (ti_n^k(x),t)$. We identify $\hat{\Delta}_n$ with its germ at $(0, 0)$.

\begin{definition}  Given a definable germ $(X, x_0)$.
A linearly vertex approaching $n$-simplex (l.v.a. $n$-simplex) in $(X, x_0)$ is a subanalytic continuous map germ : $\sigma: \hat{\Delta}_n \to (X, x_0)$ such that there is $K \geq 1$ satisfying
$$ \frac{1}{K}t \leq \|\sigma(tx, t) - x_0\| \leq Kt, \forall (x,t) \in \hat{\Delta}_n.$$

A linear vertex approaching $n$-chain in $(X,x_0)$ is a finite formal sum $\sum_{i\in I} a_i \sigma_i$ where $a_i \in \mathbb{Z}$ and $\sigma_i$ is a l.v.a. $n$-simplex in $(X, x_0)$. Denote by $MDC^{pre}_n (X, x_0)$ the abelian group of $n$-chains. The boundary of $\sigma$ is a formal sum of $(n-1)$-simplices defined as follows:
$$\partial \sigma : = \sum_{k=0}^n (-1)^k \sigma \circ j_n^k.$$
\end{definition}
\begin{definition} 
A homological subdivision of $\hat{\Delta}_n$ is a finite collection $\{\rho_i\}_{i\in I}$ of injective l.v.a. map germs $\rho_i: \hat{\Delta}_n \to \hat{\Delta}_n$ such that there is a subanalytic triangulation $\alpha: |K| \to \hat{\Delta}_n$ with the following properties
\begin{itemize}
    \item $\alpha$ is compatible with faces of $\hat{\Delta}_n$;
    \item the collection $\{T_i\}$ of maximal triangles of $\alpha$ is also indexed by $I$;
    \item for any $i \in I$, $\rho_i(\hat{\Delta}_n) = T_i$ and the map $\alpha^{-1}\circ\rho_i$ takes faces of $\hat{\Delta}_n$ to faces of $K$.
\end{itemize}
The sign of $\rho_i$, denoted by $sign(\rho_i)$, is defined to be $1$ if $\rho_i$ is orientation preserving, and $-1$ if it is of the opposite orientation.
\end{definition}

\begin{definition}  Let $b \in (0, +\infty]$. Two $n$-simplices $\sigma_1$ and $\sigma_2$ in $MDC_n^{pre}(X, x_0)$ are called $b$-equivalent (write $\sigma_1 \sim_b \sigma_2$) if 
\begin{itemize}
\item for $b < \infty$:  
$$\lim_{t\to 0} \frac{\max \{d_1(\sigma_1(tx,t), \sigma_2(tx,t)), x\in \Delta_n\}}{t^b}=0.$$

\item for $b = \infty$: $\sigma_1(tx,t)= \sigma_2(tx,t), \forall x \in \Delta_n$.
\end{itemize}
\end{definition}

\begin{definition}   Let $b \in (0, +\infty]$.
Let  $z = \sum_{i\in I} a_i \sigma_i$ and $z' = \sum_{j\in J} b_j \tau_j$  be l.v.a. $n$-complex chains in $MDC_n^{pre}(X, x_0)$. Write $I = \bigsqcup_{k\in K}I_k$ and $J = \bigsqcup_{k\in K}J_k$ where

(a) $i_1, i_2 \in I$ belong to the same $I_k$ iff $\sigma_{i_1}\sim_b \sigma_{i_2}$.
    
(b) $j_1, j_2 \in J$ belong to the same $J_k$ iff $\tau_{j_1}\sim_b \tau_{j_2}$.

(c) for $k \in K$, $i \in I_k$, $j \in J_k$, we have $\sigma_i \sim_b \tau_j$.

Then, $z$ is called {\it $b$-equivalent} to $z'$, denoted $z \sim_b z'$, if for any $k$, 
$$ \sum_{i\in I_k} a_i = \sum_{j\in J_k} b_j.$$
\end{definition}

\begin{definition} 
Let $b \in (0, +\infty]$. Given two l.v.a. $n$-complex chains $z = \sum_{i\in I} a_i \sigma_i$ and $z' = \sum_{j\in J} b_j \tau_j$ in $MDC_n^{pre}(X, x_0)$. 

\begin{itemize}
    \item We write $z \rightarrow_b z'$ (immediate relation) if for each $i \in I$ there is a homological subdivision $\{\rho_{i,k}\}_{k\in K_i}$ such that 
$$ \sum_{i\in I}\sum_{k\in K_i} sgn(\rho_{i,k}) a_i \sigma\circ \rho_{i,k} \sim_b \sum_{j\in J} b_j \tau_j.$$
\item $z$ and $z'$ are called {\it homological subdivision equivalent}, denoted by $z \sim_{S, b} z'$, if there exist sequences of immediate sequences $z = z_0 \rightarrow_b z_1 \rightarrow_b  \ldots \rightarrow_b z_l$ and $z' = w_0 \rightarrow_b  w_1 \rightarrow_b  \ldots \rightarrow_b w_m$ such that $w_m \sim_b z_l$.
\end{itemize}
\end{definition}
\begin{definition} 
Let $b \in (0, +\infty]$. The $b$-moderately discontinuous chain complex in $(X, x_0)$ is the quotient group $MDC_{\bullet}^{b}(X, x_0) := MDC_{\bullet}^{pre}(X, x_0)/ \sim_{S, b}$.  Its homology is called $b$-moderately discontinuous homology, denoted by $MDH_{\bullet}^{b} (X, x_0)$.
\end{definition}

\begin{definition}
\label{def:b-horn}
Let $X$ be a subanalytic germ at $0$ in $\R^n$. Let $b \in (0, +\infty)$. The {\em $b$-horn neighborhood of amplitude $\eta$ of $X$} is the subset 
$$
\mathcal{H}_{b,\eta}(X):=\bigcup_{x\in X}\{z\in \R^n;\|z-x\|<\eta \|x\|^b\}.
$$
The $\infty$-horn neighborhood $\mathcal{H}_{\infty,\eta}(X)$ is defined to be $X$.  
\end{definition}

We have the following two interesting results:

\begin{theorem}[Theorem 13.5 in \cite{FHPS}]
\label{th:speed1}
Let $(X,x_0)\subset (\R^m,x_0)$ be a closed subanalytic germ. Then there is a natural isomorphism $$MDH^1_*(X,x_0)\to MDH^1_*(C(X,x_0), 0)\to H_*(C(X,x_0)\setminus\{0\}).$$
\end{theorem}

\begin{theorem}
\label{cor:birbrair}
Let $X$ be a subanalytic germ at $0$ in $\R^n$. Let $b\in(0,+\infty)$ and $\eta>0$. There exists a $b'$ satisfying $b<b'$ such that the $b$-MD homology $MDH^{b}_\bullet(X,0)$ is isomorphic, for any $b''\in (b,b')$, to the singular homology of the punctured $b''$-horn neighborhood $\mathcal{H}_{b'',\eta}(X)\setminus\{0\}$. 
\end{theorem}

\section{Main result}\label{sec:main_results}

\begin{theorem}\label{main_result}
Let $X\subset\mathbb{C}^n$ be a complex analytic set. Then the following statements are equivalent:
\begin{itemize}
 \item [(1)] $X$ is $\alpha$-H\"older regular at $0$ for all $\alpha\in (0,1)$;
 \item [(2)] $X$ is strongly H\"older regular at $0$;
 \item [(3)] $X$ is log-Lipschitz regular at $0$;
 \item [(4)] $X$ is Lipschitz regular at $0$;
 \item [(5)] $X$ is smooth at $0$.
\end{itemize}
\end{theorem}
\begin{proof}
It is clear that $(5)\Rightarrow (4)\Rightarrow (3) \Rightarrow (2) \Rightarrow (1)$. So, we only have to prove that $(1) \Rightarrow (5)$. In order to do that, we assume that $X$ is $\alpha$-H\"older regular at $0$ for all $\alpha\in (0,1)$. Let $d$ be the dimension of $X$.

Let $\phi\colon X\to B$ be a bi-$\alpha$-H\"older homeomorphism, where $B$ is an open ball of $\C^d$. By McShane's Theorem (see \cite[Corollary 1]{Mcshane:1934}), there exist $\alpha$-H\"older maps $\widetilde{\phi}\colon \C^n\to \C^d$ and $\widetilde{\psi}\colon \C^d\to \C^n$ such that $\widetilde{\phi}|_X=\phi$ and $\widetilde{\psi}|_Y=\phi^{-1}$.
Let us define $\varphi, \psi\colon \C^n\times \C^d\to \C^d\times \C^n$ as follows:
$$
\varphi(x,y)=(x-\widetilde{\psi}(y+\widetilde{\phi}(x)),y+\widetilde{\phi}(x))
$$
and
$$
\psi(z,w)=(z+\widetilde{\psi}(w), w-\widetilde{\phi}(z+\widetilde{\psi}(w))).
$$
It is easy to check that $\psi$ and $\varphi$ are $\alpha^2$-H\"older maps, $\psi=\varphi^{-1}$ and $\varphi(X\times \{0\})=\{0\}\times B$. 

Thus, by doing the identification:
$X \leftrightarrow X\times\{ 0\}$ 
one can suppose that for each $\alpha\in(0,1)$ there exists a bi-$\alpha$-H\"older homeomorphism $\varphi\colon (\C^N,0)\to(\C^N,0)$ such that $\varphi(X)=\{0\}\times B$, where $N=n+d$.

Let $f_1,...,f_k\in \C[z_1,...,z_N]$ be homogeneous polynomials such that $C(X,0)=V(f_1,...,f_k)=\{z\in \C^N;f_1(z)=...=f_k(z)=0\}$. By changing $X$ by $X\times \C^M$, for some $M>0$, we may assume that $N\geq k+4$.

Let $X_0=C(X,0)\cap \mathbb{S}^{2N-1}$, and for $t>0$ let $X_t=(\frac{1}{t}X)\cap \mathbb{S}^{2N-1}$.

\begin{claim}\label{trivial_zero_homology}
 $X_0$ is a connected set.
\end{claim}
\begin{proof}[Proof of Claim \ref{trivial_zero_homology}]
Since $Link_0(X)$ is a connected set, for any $t>0$ smaller than the Milnor radius of $X$ at $0$, $\varepsilon_0$, $X_t$ is a connected set. 

We know that $\lim X_t=X_0$ with respect to the Hausdorff limit. It is well-known that the Hausdorff limit of a family of connected sets is a connected set as well. Since we could not find a reference for it, we present here a proof that $X_0$ is a connected set. 
Indeed, let $(A,B)$ be a separation of $X_0$, i.e., $A$ and $B$ are disjoint closed subsets of $X_0$ such that $X_0=A\cup B$. 
Since $X_0$ is a compact subset, we have that $A$ and $B$ are disjoint compact subsets, and thus we can find $\varepsilon >0$ such that $A_{\varepsilon}=\{ x\in \mathbb{C}^N; dist(x,A)< \varepsilon \}$ and $B_{\varepsilon}=\{ x\in \mathbb{C}^N; dist(x,B)<\varepsilon \}$ are disjoint open subsets. 
Thus, there exists $0<t_0\leq \varepsilon_0$ such that $X_t\subset U=A_{\varepsilon}\cup B_{\varepsilon}$ for all $0<t\leq t_0$. Then $(A_{\varepsilon}\cap X_t, B_{\varepsilon}\cap X_t)$ is a separation of $X_t$ for all $0<t\leq t_0$. Since $X_t$ is a connected set, $X_t\subset A_{\varepsilon}$ or $X_t\subset B_{\varepsilon}$ for all $0<t\leq t_0$. This shows that $A$ or $B$ is the empty set. Therefore $X_0$ is a connected set.
\end{proof}

\begin{claim}\label{trivial_homology}
$H_i(X_0)=0$ for all $1\leq i\leq 2d-2$.
\end{claim}
\begin{proof}[Proof of Claim \ref{trivial_homology}]
In order to prove this claim, we are going to use the moderately discontinuous homology. By Theorem \ref{th:speed1}, there is an isomorphism between $MDH^{1}_*(X,0)$ and the singular homology $H_*(X_0)$. Then, we have to prove that $MDH^{1}_i(X,0)=0$  for all $1\leq i\leq 2d-2$. 

By Theorem \ref{cor:birbrair}, for each $\eta>0$, there exists a $b'$ satisfying $1<b'$ such that the $1$-MD homology $MDH^{1}_\bullet(X,0)$ is isomorphic, for any $b''\in (b,b')$, to the singular homology of the punctured $b''$-horn neighborhood $\mathcal{H}_{b'',\eta}(X)\setminus\{0\}$. In fact, $b'$ does not depend on $\eta$ (see \cite[Remark 11.11]{FHPS}).
Fix $i\in \{1,...,2d-2\}$ and take $\xi\in H_i(Y\setminus \{0\})$, where $Y=\mathcal{H}_{b'',\eta}(X)$.
Let $\varphi\colon (\C^N,0)\to(\C^N,0)$ be a bi-$\alpha$-H\"older homeomorphism such that $\varphi(X)=\{0\}\times B$, for some open ball $B\subset \C^d$ centered at $0$. Let $C>0$ be a H\"older constant of $\varphi$ and $\varphi^{-1}$.
Then $\varphi(|\xi|)\subset \mathcal{H}_{\alpha^2b'',C\eta^{\alpha}}(\{0\}\times B)\setminus\{0\}$. Since $\mathcal{H}_{\alpha^2b'',C\eta^{\alpha}}(\{0\}\times B)\setminus\{0\}$ and $B\setminus \{0\}$ has the same homotopy type and $H_i(B\setminus\{0\})=0$, there is a chain $\tilde \tau$ in $\mathcal{H}_{\alpha^2b'',C\eta^{\alpha}}(\{0\}\times B)\setminus\{0\}$ such that $\partial \tilde \tau= \varphi_*(\xi)$. Therefore $\xi$ is homologous to zero in $\mathcal{H}_{\alpha^4b'',C^{1+\alpha}\eta^{\alpha^2}}(X)\setminus\{0\}$.

By taking $\alpha$ such that $\alpha^4b''\in (1,b')$, we obtain that $\xi$ is homologous to zero in $\mathcal{H}_{b'',\eta}(X)\setminus\{0\}$. Therefore $H_i(\mathcal{H}_{b'',\eta}(X)\setminus\{0\})=0$, and  thus $H_i(X_0)=0$.
\end{proof}

Therefore, $\widetilde{H}_i(X_0)=0$ for all $0\leq i\leq 2d-2$. In particular, $H^i(X_0)=0$ for all $1\leq i\leq 2d-2$.

Let $\pi\colon V^*\to\widetilde{V}$ denote the quotient map by the $\mathbb{C}^*$-action, where $V=C(X,0)$ and $V^*=V\setminus\{0\}$.

Thus, we have the following generalized version of Prill's Theorem (see \cite{Prill:1967}).
\begin{claim}\label{generalized_prill_thm}
$V$ is a linear subspace. 
\end{claim}
\begin{proof}[Proof of Claim \ref{generalized_prill_thm}]
Remind that $V=V(f_1,...,f_k):=\{z\in \C^N;f_1(z)=...=f_k(z)=0\}$ and $N\geq k+4$. 

Let $\mathbb{P}=\C P^{N-1}$ and $U=\mathbb{P}\setminus \widetilde{V}$.
 By the Alexander Duality, 
 $$
 H^s(\mathbb{P},\widetilde{V})=H_{2(N-1)-s}(U).
 $$

We are going to prove that $H_m(U)=0$ for all $m\geq N-1+k$. We proceed by induction on $k$. Indeed, if $k=1$, then this follows from Andreotti–Frankel theorem (see \cite{AndreottiF:1959}), since $\mathbb{P}\setminus \mathbb{P}V(f_1)$ is a smooth affine hypersurface of complex dimension $N-1$. 

Assume that $H_m(\mathbb{P}\setminus \mathbb{P}V(g_1,...,g_{k-1}))=0$ for all $m\geq N-1+k-1$, and any homogeneous polynomials $g_1,...,g_{k-1}\in\C[z_1,...,z_N]$. We have that $U=U'\cup U_k$, where $U_k=\mathbb{P}\setminus \mathbb{P}V(f_k)$ and $U'=\mathbb{P}\setminus \mathbb{P}V(f_1,...,f_{k-1})$. By hypothesis of induction, $H_m(\mathbb{P}\setminus \mathbb{P}V(f_kf_1,...,f_kf_{k-1}))=0$ for all $m\geq N-1+k-1$. By the Mayer-Vietoris sequence, we obtain $H_m(\mathbb{P}\setminus \mathbb{P}V(f_1,...,f_{k}))=0$ for all $m\geq N-1+k$.

Since $N\geq k+4$, then $H^2(\mathbb{P},\widetilde{V})=H^3(\mathbb{P},\widetilde{V})=0$, and thus the morphism $j^2\colon H^2(\mathbb{P})\to H^2(\widetilde{V})$, induced by the inclusion $j\colon \widetilde{V}\to \mathbb{P}$, is an isomorphism.

We know that the integral cohomology algebra $H^{\bullet}(\mathbb{P})$ is a truncated polynomial algebra $H^{\bullet}(\mathbb{P})=\mathbb{Z}[\alpha]/(\alpha^{N})$ generated by an element $\alpha$ of degree 2 (see \cite[Chapter 5, Proposition 1.6]{Dimca:1992}).
Thus, $\alpha_V=j^2(\alpha)$ is a generator of $H^2(\widetilde{V})$.

The associated Gysin sequence in cohomology lead to the following 
\[\begin{tikzcd}
   ... \rightarrow H^{2p+1}(X_0)\rightarrow  H^{2p}(\widetilde{V})\ar{r}{\psi_p}&H^{2p+2}(\widetilde{V})\rightarrow  H^{2p+2}(X_0)\rightarrow ...
\end{tikzcd}\]
where $\psi_p$ is the cup product with $\alpha_V$. Since $H^i(X_0)=0$ for all $1\leq i\leq 2d-2$, we obtain $\psi_p$ is an isomorphism for all $p\in\{1,...,d-2\}$ and, in particular, $H^{2d-2}(\widetilde{V})=\mathbb{Z}$ and $\alpha_V^{d-1}$ is a generator of $H^{2d-2}(\widetilde{V})$. This shows also that $\widetilde{V}$ is an irreducible algebraic set.

Let $E_m$ be a generic hyperplane in $\mathbb{P}$ of dimension $m$. Then ${\rm deg}(\widetilde{V})=\#(\widetilde{V}\cap E_{N-d})$ and $[\widetilde{V}]={\rm deg}(\widetilde{V})[E_{d-1}]$ in $H^{2(d-1)}(\mathbb{P})$, where $[A]$ denotes the fundamental class of $A$. Hence,
$$
\langle j^{2(d-1)}(\alpha^{d-1}),[\widetilde{V}]\rangle=\langle \alpha_V^{d-1},[\widetilde{V}]\rangle=\langle \alpha_V^{d-1},{\rm deg}(\widetilde{V})[E_{d-1}]\rangle={\rm deg}(\widetilde{V}).
$$
Therefore $\alpha_V^{d-1}={\rm deg}(\widetilde{V}) \cdot g$, where $g$ is a generator of $H^{2d-2}(\widetilde{V})$. Since $\alpha_V^{d-1}$ is also a generator of $H^{2d-2}(\widetilde{V})$, we obtain that ${\rm deg}(\widetilde{V})=1$, and therefore $V$ is a linear subspace.

\end{proof}

Let us choose linear coordinates $(x,y)$ such that $C(X,0)=\{(x,y)\in \mathbb{C}^N;y=0\}\cong \mathbb{C}^d$.

Let $P\colon \mathbb{C}^N\to \mathbb{C}^d$ be the mapping given by $P(x,y)=x$.

\begin{claim}\label{p_proper}
There exist a constant $R>0$ and a polydisc $\Delta\subset \mathbb{C}^N$ centred at $0$ such that $R\|P(x,y)\|\geq \|y\|$ for all $(x,y)\in X\cap \Delta$.
\end{claim}
\begin{proof}[Proof of Claim \ref{p_proper}]
Assume that Claim \ref{p_proper} is not true. Then there is a sequence $\{(x_j,y_j)\}_j\subset X$ such that $j\|x_j\|<\|y_j\|$ for all $j\in \mathbb{N}$ and $\lim\limits_{j\to +\infty}(x_j,y_j)=0$. In particular, $\lim\limits_{j\to +\infty}\frac{x_j}{\|(x_j,y_j)\|}=0$. By taking subsequence, if necessary, we may assume that $\lim\limits_{j\to +\infty}\frac{(x_j,y_j)}{\|(x_j,y_j)\|}=(0,v)$, with $\|v\|=1$. Then, $(0,v)\in P^{-1}(0)\cap C(X,0)$, which is a contradiction with $P^{-1}(0)\cap C(X,0)=\{0\}$.

Therefore Claim \ref{p_proper} holds true.
\end{proof}

Thus, $p=P|_{\Delta\cap X}\colon \Delta\cap X\to \Delta'=\pi(\Delta\cap X)$ is a proper mapping. Then, there are a complex analytic set $\Sigma \subset \Delta'$ with $\dim \Sigma <d$ and a positive integer number $k$ such that $p\colon \Delta\cap X \setminus p^{-1}(\Sigma) \to \Delta'\setminus \Sigma$ is an analytic $k$-sheeted covering mapping.

\begin{claim}\label{degree_one}
$k=1$.
\end{claim}
\begin{proof}[Proof of Claim \ref{degree_one}]
 
Let $v\in C(X,0)\setminus (\Sigma\cup C(\Sigma,0))$ be a unit vector. Then there exists $\varepsilon>0$ such that the arc $\beta\colon [0,\varepsilon)\to C(X,0)$ given by $\beta(t)=tv$ and satisfies $\beta(t)\not \in \Sigma \cup C(\Sigma,0)$ for all $t\in (0,\varepsilon)$. In fact, there exist positive real numbers $\eta$ and $\delta$ such that
$$C_{\eta,\delta}=\{v\in C(X,\infty);\ \|v-tv_0\|< \eta t, \mbox{ for some } t\in (0,\delta) \}$$
does not intersect the set $\Sigma$.

Assume that $k>1$. Then, there are two different arcs $\gamma_1,\gamma_2\colon [0,\varepsilon)\to X$ such that $p\circ \gamma_i(t)=\beta(t)$ for all $t\in [0,\varepsilon)$, $i=1,2$.

\begin{claim}\label{b_bigger_one}
$b:=ord_0\|\gamma_1(t)-\gamma_2(t)\|>1$.
\end{claim}
\begin{proof}[Proof of Claim \ref{b_bigger_one}]
We have to show that $\lim\limits_{t\to 0^+}\frac{\|\gamma_1(t)-\gamma_2(t)\|}{t}=0$.

In fact, let us write $\gamma_i(t)=(tv,y_i(t))$. By Claim \ref{p_proper}, $\frac{\|y_i(t)\|}{t}\leq R$. So, the limit $\lim\limits_{t\to 0^+}\frac{y_i(t)}{t}$ exists, let us say, $\lim\limits_{t\to 0^+}\frac{y_i(t)}{t}=w_i$. Then, $\lim\limits_{t\to 0^+}\frac{\gamma_i(t)}{t}=(v,w_i)\in C(X,0)$, for $i=1,2.$ Since $C(X,0)=\{(x,y)\in \mathbb{C}^N;y=0\}$, we obtain that $w_1=w_2=0$, and thus $\lim\limits_{t\to 0^+}\frac{\gamma_1(t)}{t}=\lim\limits_{t\to 0^+}\frac{\gamma_2(t)}{t}$, which finishes the proof of Claim \ref{b_bigger_one}.
\end{proof}
Let $d_{X,P}$ be a pancake distance on $X$ which is equivalent to $d_{X,inn}$ (see Proposition \ref{prop:pancake_distance}). Since $d_{X,P}$ is a subanalytic distance, $ord_0 d_{X,P}(\gamma_1(t),\gamma_2(t))$ is a well-defined positive rational number, and thus we define 
$$ord_0d_{X,inn}(\gamma_1(t),\gamma_2(t)):=ord_0 d_{X,P}(\gamma_1(t),\gamma_2(t)).$$
\begin{claim}\label{a_one}
$ord_0d_{X,inn}(\gamma_1(t),\gamma_2(t))=1$.
\end{claim}
\begin{proof}[Proof of Claim \ref{a_one}]
Let $\beta_t\colon [0,1]\to X$ be an arc connecting $\gamma_1(t)$ and $\gamma_2(t)$. Thus, $p\circ \beta_t$ is a loop in $\mathbb{C}^d$ based at $tv$. Since $p\colon \Delta\cap X \setminus p^{-1}(\Sigma) \to \Delta'\setminus \Sigma$ is a covering mapping and $C_{\eta,\delta}$ is simply connected subset of $\Delta'\setminus \Sigma$, the image of $p\circ \beta_t$ can not be contained in $C_{\eta,\delta}$, and thus $\ell(p\circ \beta_t)\geq 2\eta t$. By using the facts that $p$ is 1-Lipschitz and that $1$-Lipschitz mappings do not increase the length of arcs, we obtain
$$
\ell(\beta_t)\geq \ell(p\circ \beta_t)\geq 2\eta t.
$$
By taking the infimum over all $\beta_t$, we obtain $d_{X,inn}(\gamma_1(t),\gamma_2(t))\geq 2\eta t$. Therefore $ord_0 d_{X,inn}(\gamma_1(t),\gamma_2(t))=1$.
\end{proof}

Let $\alpha\in (0,1)$ be a number such that $\alpha^2>\frac{1}{b}$.

By hypothesis, there are a neighbourhood of $0$, $U\subset \mathbb{C}^N$, $r>0$ and a bi-$\alpha$-H\"older homeomorphism $\psi\colon B_r^{2d}(0)\to X\cap U$. Thus, there is $M>0$ such that 
$$
\frac{1}{M}\|x-y\|^{\frac{1}{\alpha}}\leq \|\psi(x)-\psi(y)\|\leq M\|x-y\|^{\alpha}, \quad \forall x,y\in B_r^{2d}(0).
$$
This implies that $d_{X,diam}(u,v)\leq M^2\|u-v\|^{\alpha^2}$ for all small enough $u,v \in X$.

It follows from \cite[Proposition 3]{KurdykaO:1997} (see also Subsection \ref{subsec:pancake_dec}) that we can write $X=\bigcup\limits_{i=1}^m X_i$ such that each $X_i$ is an LNE closed set. Let $C$ be a positive number such that each $X_i$ is $C$-LNE for all $i\in\{1,...,m\}$.

\begin{claim}\label{claim:diameter_inner_equivalence}
$d_{X,inn}(x,y)\leq (mC)\cdot d_{X,diam}(x,y)$ for all $x,y\in X$. 
\end{claim}
\begin{proof}[Proof of Claim \ref{claim:diameter_inner_equivalence}]
Fix $x,y\in X$ and let $\alpha\colon [0,1]\to X$ be an arc connecting $x$ and $y$. 

We define $s_0=0$. Let $i_1$ such that $\gamma_1=\alpha(0)\in X_{i_1}$. Let $s_1=\sup\{s\in [0,1];\alpha(s)\in  X_{i_1}\}$. Since $ X_{i_1}$ is a closed set, $\alpha(s_1)\in  X_{i_1}$. Assume that $s_j$ and $i_j$ are defined. If $s_j=1$, we stop the process. If $s_j<1$, let $i_{j+1}\not\in \{i_1,...,i_j\}$ such that $\alpha(s_j)\in  X_{i_{j+1}}$. We define $s_{j+1}=\sup\{s\in [s_j,1];\alpha(s)\in  X_{i_{j+1}}\}$. Since $ X_{i_{j+1}}$ is a closed set, $\alpha(s_{j+1})\in  X_{i_{j+1}}$. Since we only have a finite number of $X_i$'s, there is a $j_0$ such that $s_{j_0}=1$. For each $j\in \{0,...,j_0-1\}$, we consider $\beta_j\colon [0,1]\to X_{i_{j+1}}$ be a minimizing geodesic connecting $\alpha(s_{j})$ and $\alpha(s_{j+1})$. Let $\beta\colon [0,1]\to X$ be the broken geodesic $\beta_0\cdot ... \cdot  \beta_{j_0-1}$.

Since $diam(Im(\alpha))\geq diam(Im(\alpha)\cap X_{i_{j+1}})\geq \|\alpha(s_{j+1})-\alpha(s_{j})\|$ for all $j\in \{0,...,j_0-1\}$, we obtain
$$
m\cdot diam(Im(\alpha))\geq j_0\cdot diam(Im(\alpha))\geq \frac{1}{C}\ell (\beta)\geq \frac{1}{C} d_{X,inn}(x,y).
$$

By taking the infimum over all $\alpha$, we obtain 
$$
d_{X,inn}(x,y) \leq (mC)\cdot d_{X,diam}(x,y).
$$
\end{proof}
Then $d_{X,inn}(\gamma_1(t),\gamma_2(t))\leq mCM^2  \|\gamma_1(t)-\gamma_2(t)\|^{\alpha^2}$ for all small enough $t>0$, and thus $1\geq \alpha^2b$, which is a contradiction with the chosen $\alpha$.

Therefore $k=1$.
\end{proof}

Since $k=1$, we have that $p\colon \Delta\cap X \setminus p^{-1}(\sigma) \to \Delta'\setminus \sigma$ is an analytic diffeomorphism. Let $\Phi\colon \Delta'\setminus  \Sigma \to X\setminus p^{-1}(\Sigma)$ be its inverse. Thus, there is an analytic mapping $\phi\colon \Delta'\setminus  \Sigma \to \mathbb{C}^{n-d}$ such that $\Phi(x)=(x,\phi(x))$ for all $x\in \Delta'\setminus  \Sigma$. By Claim \ref{p_proper}, there exists a constant $R>0$ such that $R\|p(x,y)\|\geq \|y\|$ for all $(x,y)\in X\cap \Delta$, then $\phi$ is a bounded mapping on $\Delta'\setminus  \Sigma$. Since $\Sigma$ is a removable set, $\phi$ can be extended to a complex analytic mapping $\tilde \phi\colon \Delta'\to \mathbb{C}^{n-d}$. Therefore, $X\cap \Delta$ is the graph of the analytic mapping $\tilde\phi$ and, in particular, $X$ is smooth at $0$.
\end{proof}

The next result says that Theorem \ref{main_result} is sharp in the sense that one cannot change the item $(1)$ in Theorem \ref{main_result} by the following weaker statement: $X$ is $\alpha$-H\"older regular at $0$, for some $\alpha\in (0,1)$ that only depends on the dimension of $X$.

\begin{proposition}\label{prop:example_sharp}
For any $\alpha\in (0,1)$ and any positive integer $d$, there is a complex algebraic hypersurface $X\subset \C^{d+1}$ that is $\alpha$-H\"older regular at $0$, but is not smooth at $0$.
\end{proposition}
\begin{proof}
For $\alpha\in (0,1)$, let $n>1$ be a integer such that $\alpha<\frac{n}{n+1}$.

Let $X=\{(x,y)\in \C^2; x^n=y^{n+1}\}$ and $L=X\cap \mathbb{S}^3$. Since $X$ is an irreducible and weighted-homogeneous curve, then there is a diffeomorphism $f\colon  \mathbb{S}^1\rightarrow L$. In particular, $f$ is a bi-Lipschitz homeomorphism, i.e., there is $\lambda \geq 1$ such that  $$\frac{1}{\lambda} \| v_1-v_2 \| \leq \| f(v_1)-f(v_2) \| \leq \lambda \| v_1 - v_2 \| \ \forall \ v_1,v_2\in \mathbb{S}^1.$$ 
We write $f(v)=(f_1(v),f_2(v))$ and define $F\colon \C\rightarrow X$ by 
$$
F(t v) = (t f_1(v),t^{\frac{n}{n+1}} f_2(v))\quad \mbox{ for all  }v\in \mathbb{S}^1 \mbox{ and }t\geq 0.
$$
We claim that $F$ is a $b$-H\"older mapping around $0$, where $b=\frac{n}{n+1}$. Indeed, given $u_1=t v_1, u_2=s v_2\in \C$ with $v_1,v_2\in\mathbb{S}^1$ and $t\leq s\leq 1$, we have
\begin{eqnarray*}
	\| F(u_1) - F(u_2) \| &\leq & \| F(t v_1) - F(t v_2)\| + \| F(t v_2) - F(s v_2)\| \\
	 &\leq& t \| f_1(v_1) - f_1(v_2) \|+t^b \| f_2(v_1) - f_2(v_2) \|   \\
	 & & + |t-s|\| f_1(v_2) \| +|t^b-s^b|\| f_2(v_2) \|\\
	 &\leq& \lambda t \| v_1 - v_2 \| +\lambda t^b \| v_1 - v_2 \|+ |t-s| \|v_2\|+|t^b-s^b|\| v_2 \| \\
	 &\leq& (\lambda+1) \| t v_1 - s v_2 \| +2^{1-b}\lambda\| t v_1 - t v_2 \|^b+ |t-s|^b\| v_2 \| \\
	 &\leq& (2\lambda+1) \| u_1 - u_2 \| + (2^{1-b}\lambda+1)\| t  v_1 -s v_2\|^b \\
	 &\leq& (6\lambda+3) \| u_1 -u_2\|^b.
\end{eqnarray*}

If we denote by $g\colon L\rightarrow \mathbb{S}^1$ the inverse mapping of $f$, we have that $G\colon X \rightarrow \C$ defined by $G(tx,t^{b}y)= t g(x,y)$ and $t\geq 0$ is the inverse of $F$ and it is a Lipschitz mapping around $0$. Indeed, given $w_1=(tx_1,t^{b}y_2), w_2=(sx_2,s^{b}y_2)\in X$ with $(x_1,y_1), (x_2,y_2)\in L$ and $0\leq t\leq s\leq 1$, we have
\begin{eqnarray*}
	\| G(w_1) - G(w_2) \| &\leq & t \| g(x_1,y_1) - g(x_2,y_2) \|+ |t-s|\| g(x_2,y_2) \|\\
	 &\leq& \lambda t \| (x_1,y_1) - (x_2,y_2) \| + |t-s| \|(x_2,y_2)\| \\
	 &\leq& (\lambda+1) \| t(x_1,y_1) - s(x_2,y_2) \|\\
	 &\leq & (\lambda+1) (\| t x_1 - s x_2 \| + \| t  y_1 -s y_2\|) \\
	 &\leq & (\lambda+1)(1+\sqrt{2}) \| w_1 - w_2 \|.
\end{eqnarray*} 

Since $\alpha<b$, this implies that $F$ is a bi-$\alpha$-H\"older homeomorphism around $0$, and thus  $X$ is $\alpha$-H\"older regular at $0$.

Moreover, for any positive integer $d$, $X\times \C^{d-1}\subset \C^{d+1}$ is $\alpha$-H\"older regular at $0$, but is not smooth at $0$.

\end{proof}

\section{The diameter and the inner distances are equivalent}
Here, we prove a result that is slightly more general than Claim \ref{claim:diameter_inner_equivalence}, since here we do not assume that the set is compact.
\begin{proposition}\label{prop:diam_inner}
 Let $X$ be a path-connected set. Let $X=\bigcup_{i\in I}{B_i}$ be a finite decomposition such that each $B_i$ is $C$-LNE. Then 
 $$d_{X,diam}\leq (\#I\cdot C)d_{X,inn}.$$ 
\end{proposition}
\begin{proof}Fix $\varepsilon>0$ and let $m=\#I$. For each $i$, let $X_i$ be the closure of $B_i$ in $X$.

Fix $x,y\in X$ and let $\alpha\colon [0,1]\to X$ be an arc connecting $x$ and $y$. 

We define $s_0=0$. Let $i_1$ such that $\gamma_1=\alpha(0)\in X_{i_1}$. Let $s_1=\sup\{s\in [0,1];\alpha(s)\in  X_{i_1}\}$. Since $ X_{i_1}$ is a closed set in $X$, $\alpha(s_1)\in  X_{i_1}$. Assume that $s_j$ and $i_j$ are defined. If $s_j=1$, we stop the process. If $s_j<1$, let $i_{j+1}\not\in \{i_1,...,i_j\}$ such that $\alpha(s_j)\in  X_{i_{j+1}}$. We define $s_{j+1}=\sup\{s\in [s_j,1];\alpha(s)\in  X_{i_{j+1}}\}$. Since $ X_{i_{j+1}}$ is a closed set, $\alpha(s_{j+1})\in  X_{i_{j+1}}$. Since we only have a finite number of $X_i$'s, there is a $j_0$ such that $s_{j_0}=1$. For each $j\in \{0,...,j_0-1\}$, we consider $\beta_j\colon [0,1]\to X_{i_{j+1}}$ be an arc connecting $\alpha(s_{j})$ and $\alpha(s_{j+1})$ such that $length(\beta_j)\leq (1+\varepsilon)d_{X,inn}(\alpha(s_{j}),\alpha(s_{j+1}))$. Let $\beta\colon [0,1]\to X$ be the arc $\beta_0\cdot ... \cdot  \beta_{j_0-1}$.

Since $diam(Im(\alpha))\geq diam(Im(\alpha)\cap X_{i_{j+1}})\geq \|\alpha(s_{j+1})-\alpha(s_{j})\|$ for all $j\in \{0,...,j_0-1\}$, we obtain
$$
m\cdot diam(Im(\alpha))\geq j_0\cdot diam(Im(\alpha))\geq \frac{1}{C(1+\varepsilon)}\ell (\beta)\geq \frac{1}{C(1+\varepsilon)} d_{X,inn}(x,y).
$$

By taking the infimum over all $\alpha$, we obtain 
$$
d_{X,inn}(x,y) \leq (mC(1+\varepsilon))\cdot d_{X,diam}(x,y)
$$
and setting $\varepsilon\to 0^+$, we have
$$
d_{X,inn}(x,y) \leq (mC)\cdot d_{X,diam}(x,y).
$$
\end{proof}
Since any o-minimal set can be partitioned into a finite union of definable LNE sets (e.g., see \cite[Theorem 1.3]{KurdykaP:2006}), we obtain the following result, which gives a positive answer to  Problem \ref{conjecture:diameter_inner_equivalence}. 
\begin{theorem}\label{thm:diameter_inner_equivalence}
 Let $X$ be a connected definable set in an o-minimal structure on $\R$. Then $d_{X,diam}$ is equivalent to $d_{X,inn}$. 
\end{theorem}

\section{A log-Lipschitz version of a result of Ahern and Rudin}

\begin{theorem}\label{thm:gen_Ahern-Rudin}
Let $V\subset \mathbb{C}^n$ be a complex algebraic set. Then the following statements are equivalent:
\begin{itemize}
 \item [(1)] $V$ is $\alpha$-H\"older smooth at infinity for all $\alpha\in (0,1)$;
 \item [(2)] $V$ is strongly H\"older smooth at infinity;
 \item [(3)] $V$ is log-Lipschitz smooth at infinity;
 \item [(4)] $V$ is Lipschitz smooth at infinity;
 \item [(5)] $V$ is the union of an affine linear subspace of $\mathbb{C}^n$ and a (possibly empty) finite set.
\end{itemize}
\end{theorem}
\begin{proof}
It is clear that $(5)\Rightarrow (4)\Rightarrow (3) \Rightarrow (2) \Rightarrow (1)$. So, we only have to prove that $(1) \Rightarrow (5)$.

Assume that $V$ is $\alpha$-H\"older smooth at infinity for all $\alpha\in (0,1)$. In this case, $V$ is the union of a pure dimensional algebraic set and a (possibly empty) finite set. Thus, we may assume that $V$ is a pure dimensional set.

Let $A=\iota(V\setminus\{0\})\cup \{0\}$, where $\iota\colon \C^n\setminus \{0\}\to \C^n\setminus \{0\}$ is the mapping given by $\iota(x)=\frac{x}{\|x\|^2}$. It is easy to check that $C(V,\infty)=C(A,0)$. Then $C(A,0)$ is a pure $d$-dimensional complex homogeneous algebraic set, where $d=\dim_{\C}V$. 
By Claims \ref{trivial_zero_homology} and \ref{trivial_homology}, $\widetilde{H}_i(A_0)=0$ for all $i\in \{0,...,2d-2\}$, where $A_0=C(A,0)\cap \mathbb{S}^{2n-1}$.

By Claim \ref{generalized_prill_thm}, $C(A,0)=C(V,\infty)$ must be a complex linear subspace. 

By following the proof of Theorem \ref{main_result}, if $V$ is not an affine linear subspace, then for a generic $v\in C(V,\infty)$ there are two different arcs $\gamma_1,\gamma_2\colon (\varepsilon, \infty)\to V$ such that $p(\gamma_j(t))=tv$ for all $t\in (\varepsilon, \infty)$ and $j=1,2$, where $p$ is the restriction to $V$ of the orthogonal projection onto $C(V,\infty)$. Moreover, 
$$
ord_{\infty} d_{V,inn}(\gamma_1,\gamma_2):=\lim\limits_{t\to+\infty}\frac{\log d_{V,inn}(\gamma_1(t),\gamma_2(t))}{\log t}=1
$$
and $\lim\limits_{t\to+\infty}\frac{\gamma_1(t)-\gamma_2(t)}{t}=0$, and thus we have that
$$
b=ord_{\infty} \|\gamma_1-\gamma_2\|:=\lim\limits_{t\to+\infty}\frac{\log \|\gamma_1(t)-\gamma_2(t)\|}{\log t}<1.
$$
For each $j\in\{1,2\}$, let $\tilde\gamma_j\colon (0,\frac{1}{\varepsilon})\to A$ be the arc given by $\tilde \gamma_j(s)=\iota\circ \gamma\left(\frac{1}{s}\right)$. Then, $ord_0  d_{A,inn}(\tilde\gamma_1,\tilde\gamma_2)=1
$ and $\tilde b=ord_0  \|\tilde\gamma_1-\tilde\gamma_2\|=2-b>1$ (see the proof of Corollary 5.7 in \cite{FernandesS:2022}).

Again, by following the proof of Theorem \ref{main_result}, for $\alpha\in (\tilde b^{-\frac{1}{2}},1)$, there is a constant $M_{\alpha}>0$ such that
$d_{A,diam}(u,v)\leq M_{\alpha}\|u-v\|^{\alpha^2}$ for all small enough $u,v \in A$. Since $A$ is a path-connected semialgebraic set, by Theorem \ref{thm:diameter_inner_equivalence}, there is a constant $C$ such that $d_{A,inn}\leq C d_{A,diam}$. In particular, $d_{A,inn}(\tilde \gamma_1(s),\tilde \gamma_2(s))\leq CM_{\alpha}  \|\tilde \gamma_1(t)-\tilde \gamma_2(t)\|^{\alpha^2}$ for all small enough $t>0$, and thus $1\geq \alpha^2\tilde b>1$, which is a contradiction. Therefore, $V$ must be an affine linear subspace.
\end{proof}

As a direct consequence, we obtain the next result, which generalizes \cite[Thereom II]{AhernR:1993}.

\begin{corollary}\label{cor:AhernR_Lip}
A complex analytic set $V\subset \mathbb{C}^n$ is Lipschitz smooth at infinity if and only if $V$ is the union of an affine linear subspace of $\mathbb{C}^n$ and a (possibly empty) finite set.
\end{corollary}
\begin{proof}
Assume that $V$ is Lipschitz smooth at infinity.
From \cite[Theorem 5.1]{Sampaio:2023b} and \cite[Theorem 3.1]{Sampaio:2023}, we obtain that $V$ is a complex algebraic set. By Theorem \ref{thm:gen_Ahern-Rudin}, $V$ is the union of an affine linear subspace of $\mathbb{C}^n$ and a (possibly empty) finite set.

The reciprocal is obvious.
\end{proof}

Ahern and Rudin in \cite{AhernR:1993} defined also the notion of a mapping to be $C^1$-smooth at infinity.
\begin{definition}
A map $F\colon \R^n\to \R^n$ is {\bf $C^1$-smooth at infinity} if $\iota \circ F\circ \iota$ extends to a $C^1$ mapping in a neighbourhood of the origin of $\R^n$.
\end{definition}
Thus, we define also the following:
\begin{definition}
A map $F\colon \R^n\to \R^n$ is {\bf Lipschitz smooth at infinity} if $\iota \circ F\circ \iota$ extends to a  Lipschitz mapping in a neighbourhood of the origin of $\R^n$.
\end{definition}

It is also clear that if $F\colon \R^n\to \R^n$ is $C^1$-smooth at infinity then it is Lipschitz smooth at infinity. Thus, we obtain the following version of \cite[Thereom I]{AhernR:1993}:

\begin{corollary}
A bi-holomorphism $F\colon \mathbb{C}^n\to \mathbb{C}^n$ is Lipschitz smooth at infinity if and only if $F$ is an affine mapping.
\end{corollary}
\begin{proof}
It is clear that any affine mapping that is a bi-holomorphism must be a bi-Lipschitz homeomorphism. Then, by \cite[Corollary 5.2]{Sampaio:2023b}, $\iota \circ F\circ \iota$ extends to a bi-Lipschitz homeomorphism $\tilde F\colon \mathbb{C}^n\to \mathbb{C}^n$. In particular, $F\colon \mathbb{C}^n\to \mathbb{C}^n$ is Lipschitz smooth at infinity.

Reciprocally, assume that $F\colon \mathbb{C}^n\to \mathbb{C}^n$ is Lipschitz smooth at infinity. By \cite[Theorem 5.1]{Sampaio:2023b}, $F\colon \mathbb{C}^n\to \mathbb{C}^n$ is Lipschitz at infinity. By Liouville's Theorem for holomorphic mappings, $F$ is an affine mapping.
\end{proof}

\end{document}